\documentclass[12pt]{amsart}

\normalfont
\usepackage[T1]{fontenc}

\usepackage{amssymb}
\usepackage{graphicx}
\usepackage{lscape}
\usepackage{mathrsfs}

\usepackage[curve,tips]{xypic}
\SelectTips{eu}{11} \UseTips

\renewcommand{\epsilon}{\varepsilon}

\newtheorem{theorem}{Theorem}[section]
\newtheorem{prop}[theorem]{Proposition}

\newtheorem{lemma}[theorem]{Lemma}

\newtheorem*{theoremA}{Theorem A}

\theoremstyle{definition}

\newtheorem{definition}[theorem]{Definition}

\theoremstyle{remark}
\newtheorem{remark}[theorem]{Remark}

\newcommand{\cll}{{\scriptstyle\bf L}}
\newcommand{\clh}{{\scriptstyle\bf H}}
\newcommand{\lhf}{\cll\clh\mathfrak F}

\newcommand{\XX}{\mathfrak X}

\newcommand{\pd}{\operatorname{proj. dim}}

\newcommand{\Z}{\mathbb Z}
\newcommand{\N}{\mathbb N}
\newcommand{\R}{\mathbb R}

\newcommand{\FP}{\operatorname{FP}}

\newcommand{\fpinfty}{{\FP}_{\infty}}

\newcommand{\flinfty}{\operatorname{FL}_{\infty}}

\newcommand{\EXT}{\operatorname{Ext}}

\newcommand{\cohom}[3]{H^{{\raise1pt\hbox{$\scriptstyle#1$}}}(#2\>\!,#3)}
\newcommand{\tatecohom}[3]{\widehat H^{{\raise1pt\hbox{$\scriptstyle#1$}}}(#2\>\!,#3)}

\newcommand{\Cohom}[3]%
  {H^{{\raise1pt\hbox{$\scriptstyle#1$}}}\big(#2\>\!,#3\big)}
\newcommand{\Tatecohom}[3]%
  {\widehat H^{{\raise1pt\hbox{$\scriptstyle#1$}}}\big(#2\>\!,#3\big)}

\newcommand{\homol}[3]{H_{{\lower1pt\hbox{$\scriptstyle#1$}}}(#2\>\!,#3)}
\newcommand{\homolog}[2]{H_{{\lower1pt\hbox{$\scriptstyle#1$}}}(#2)}

\newcommand{\colim}{\varinjlim}

\renewcommand{\ker}{\operatorname{Ker}}
\newcommand{\im}{\operatorname{Im}}
\newcommand{\coker}{\operatorname{Coker}}

\newcommand{\mono}{\rightarrowtail}
\newcommand{\epi}{\twoheadrightarrow}

\newcommand{\ra}{\rightarrow}

\newcommand{\Hom}{\operatorname{Hom}}

\newcommand{\blah}{{\phantom M}}

\newcommand{\id}{\operatorname{id}}

\newcommand{\exthat}{\widehat{\EXT}}
\newcommand{\hombar}{\underline{\Hom}}

\title{Finitary Group Cohomology and Eilenberg--Mac Lane spaces}
\author{Martin Hamilton}
\address{Hausdorff Center for Mathematics, Universität Bonn, Landwirtschaftskammer (Neubau), Endenicher Allee 60, 53115 Bonn}
\email{hamilton@math.uni-bonn.de}

\subjclass[2000]{20J06 55P20 18A22}

\keywords{cohomology of groups, finitary functors, Eilenberg--Mac
Lane spaces}

\begin{document}

\begin{abstract}
  We say that a group $G$ has \emph{cohomology almost everywhere
  finitary} if and only if the $n$th cohomology functors of $G$
  commute with filtered colimits for all sufficiently large $n$.

  In this paper, we show that if $G$ is a group in Kropholler's class $\lhf$ with
  cohomology almost everywhere finitary, then $G$ has an
  Eilenberg--Mac Lane space $K(G,1)$ which is dominated by a
  CW-complex with finitely many $n$-cells for all sufficiently
  large $n$. It is an open question as to whether this holds for
  arbitrary $G$.

  We also remark that the converse holds for any group $G$.
\end{abstract}

\maketitle

\section{Introduction}

Let $G$ be a group and $n\in\N$. The $n$th cohomology of $G$ is a
functor $H^n(G,-)$ from the category of $\Z G$-modules to the
category of abelian groups. We are interested in groups whose
$n$th cohomology functors are \emph{finitary}; that is, they
commute with all filtered colimit systems of coefficient modules.

We are concerned with the class $\lhf$ of locally hierarchically
decomposable groups (see \cite{fp} for a definition of this
class). If $G$ is an $\lhf$-group, then Theorem 2.1 in
\cite{Continuity:cohomologyfunctors} shows that the set
$$\{n\in\N:H^n(G,-)\mbox{ is finitary}\}$$ is either cofinite or finite. If this set is cofinite,
we say that $G$ has \emph{cohomology almost everywhere finitary},
and if this set is finite, we say that $G$ has \emph{cohomology
almost everywhere infinitary}.

In \cite{Finitarycohom} we investigated algebraic
characterisations of certain classes of $\lhf$-groups with
cohomology almost everywhere finitary. In this paper we prove the
following topological characterisation:

\begin{theoremA}
  Let $G$ be a group in the class $\lhf$. Then the following are
  equivalent: \begin{enumerate}
    \item $G$ has cohomology almost everywhere finitary;

    \item $G\times\Z$ has an Eilenberg--Mac Lane space $K(G\times\Z,1)$  with
    finitely many $n$-cells for all sufficiently large $n$;

    \item $G$ has an Eilenberg--Mac Lane space $K(G,1)$ which is dominated
    by a CW-complex with finitely many $n$-cells for all
    sufficiently large $n$.
  \end{enumerate}
\end{theoremA}

The implications (ii) $\Rightarrow$ (iii) and (iii) $\Rightarrow$
(i) hold for any group $G$, while our proof of (i) $\Rightarrow$
(ii) requires the assumption that $G$ belongs to the class $\lhf$.
We do not know whether (i) $\Rightarrow$ (ii) holds for arbitrary
$G$.

\subsection{Acknowledgements}

I would like to thank my research supervisor Peter Kropholler for
suggesting that a result like Theorem A should be true, and for
his advice and support throughout this project. I would also like
to thank Philipp Reinhard for explaining the arguments in Lemma
\ref{replace with one with same fun gp lemma}

\section{Proof}

\subsection{Proof of Theorem A (i) $\Rightarrow$ (ii)}$ $

Suppose that $G$ is an $\lhf$-group with cohomology almost
everywhere finitary. We need to make use of \emph{complete
cohomology}, and refer the reader to
\cite{strgpgrdrings,completeres,Mislin} for further information on
definitions etc. If $R$ is a ring, then we can consider the
\emph{stable category} of $R$-modules; the objects are the
$R$-modules and the stable maps $M\ra N$ between $R$-modules are
the elements of the complete cohomology group $\exthat^0_R(M,N)$.

We make the following definitions:

\begin{definition}
  Let $R$ be a ring. An $R$-module $M$ is said to be
  \emph{completely finitary} (over $R$) if and only if the functor
  $$\exthat^n_R(M,-)$$ is finitary for all integers $n$.
\end{definition}

\begin{remark}
  We see from $4.1$(ii) in \cite{fp} that every $R$-module of type
  $\fpinfty$ is completely finitary.
\end{remark}

\begin{definition}
  Let $R$ be a ring. An $R$-module $N$ is said to be
  \emph{completely flat} (over $R$) if and only if $$\exthat^0_R(M,N)=0$$ for
  all completely finitary $R$-modules $M$.
\end{definition}

We have a version of the Eckmann--Shapiro Lemma for complete
cohomology (Lemma $1.3$ in \cite{polyelementary}):

\begin{lemma}\label{Eckmann--Shapiro Lemma}
  Let $H$ be a subgroup of $G$, $V$ be a $\Z H$-module and
  $N$ be a $\Z G$-module. Then, for all integers $n$, there is a natural
  isomorphism $$\exthat^n_{\Z G}(V\otimes_{\Z H}\Z G,N)\cong\exthat^n_{\Z
  H}(V,N).$$
\end{lemma}

Now, let $G$ be an $\lhf$-group, and $N$ be a $\Z G$-module. To
check whether $N$ is completely flat, it is enough to check
whether the restriction of $N$ to every finite subgroup of $G$ is
completely flat, by the following proposition. This is where the
assumption that $G$ belongs to $\lhf$ is used.

\begin{prop}\label{reduction to fin subgp prop}
  Let $G$ be an $\lhf$-group, and $N$ be a $\Z
  G$-module. Then the following are equivalent: \begin{enumerate}
    \item $N$ is completely flat as a $\Z G$-module;

    \item $N$ is completely flat as a $\Z K$-module for all finite
    subgroups $K$ of $G$.
  \end{enumerate}
\end{prop}

\begin{proof}$ $

\begin{itemize}
  \item (i) $\Rightarrow$ (ii): Follows from Lemma \ref{Eckmann--Shapiro Lemma}.

  \item (ii) $\Rightarrow$ (i): An easy generalization of Proposition
  $6.8$ in \cite{polyelementary} shows that if $N$ is a $\Z
  G$-module which is completely flat as a $\Z K$-module for all finite
  subgroups $K$ of $G$, then $N\otimes_{\Z H}\Z G$ is completely
  flat as a $\Z G$-module for all $\lhf$-subgroups $H$ of $G$. Then, as
  we are assuming that $G$ belongs to $\lhf$, the result follows.
\end{itemize}\end{proof}

Write $B:=B(G,\Z)$ for the $\Z G$-module of bounded functions from
$G$ to $\Z$. Following Benson \cite{complexity1,complexity2}, a
$\Z G$-module $M$ is said to be \emph{cofibrant} if $M\otimes B$
is projective. We now make the following definition:

\begin{definition}
  Let $G$ be a group. A $\Z G$-module is called \emph{basic}
  if it is of the form $U\otimes_{\Z K}\Z G$, where $K$ is a
  finite subgroup of $G$ and $U$ is a
  completely finitary, cofibrant $\Z K$-module.

  A $\Z G$-module $M$ is called \emph{poly-basic} if it has a
  series $$0=M_0\leq\cdots\leq M_n=M$$ in which the sections
  $M_i/M_{i-1}$ are basic.
\end{definition}

The first step in the proof of Theorem A involves the following
construction, which is a variation on the construction found in \S
$4$ of \cite{polyelementary}:

\begin{definition}\label{Minfty}
  Let $G$ be a group, and $M$ be a $\Z G$-module. We
  construct a chain $$M=M_0\subseteq M_1\subseteq
  M_2\subseteq\cdots$$ inductively so that for each $n\geq 0$
  there is a short exact sequence $$C_n\mono M_n\oplus P_n\epi
  M_{n+1}$$ in which \begin{enumerate}
    \item $C_n$ is a direct sum of basic modules;

    \item $P_n$ is projective; and

    \item every map from a basic module to $M_n$ factors
    through $C_n$.
  \end{enumerate}

  Set $M_0=M$. Suppose that $n\geq 0$ and that $M_n$ has been
  constructed. Consider the pointed category whose objects are
  ordered pairs $(C,\phi)$, where $C$ is a basic module and
  $\phi$ is a homomorphism from $C$ to $M_n$, and whose morphisms
  are the obvious commutative triangles. Choose a set $\XX_n$
  containing at least one object of this category from each
  isomorphism class. Set $C_n:=\bigoplus_{(C,\phi)\in\XX_n} C$ and
  use the maps $\phi$ associated to each object to define a map
  $C_n\ra M_n$. Properties (i) and (iii) are now guaranteed.
  
  Note that any basic module $U\otimes_{\Z K}\Z G$ can be written as a direct sum of copies of $U$. Then, as tensor products commute with direct sums, we see that any basic module is itself cofibrant. Hence $C_n$ is cofibrant, so $C_n\otimes B$ is projective and we can set
  $P_n:=C_n\otimes B$. Finally, $M_{n+1}$ can be defined as the
  cokernel of this inclusion $C_n\ra M_n\oplus P_n$, or in other
  words the pushout, and since the map $C_n\ra P_n$ is an
  inclusion, it follows that the induced map $M_n\ra M_{n+1}$ is
  also injective and we regard $M_n$ as a submodule of $M_{n+1}$.
  Finally, set $M_{\infty}$ to be the colimit
  $$M_{\infty}:=\colim_n
  M_n.$$

\end{definition}

Next, we have the following technical proposition, which shall be
needed in the proof of Proposition \ref{divides p plus k prop}:

\begin{prop}\label{colimit prop}
  Let $G$ be a group, and $M$ be a $\Z
  G$-module. Construct the chain $$M=M_0\subseteq M_1\subseteq
  M_2\subseteq\cdots$$ as in Definition \emph{\ref{Minfty}}. Then for
  each $n$, we can express $M_{n+1}$ as a filtered colimit $$M_{n+1}:=\colim_{\lambda_n} \frac{M_n\oplus
  P_{\lambda_n}}{C_{\lambda_n}}$$ where $P_{\lambda_n}$ is
  projective and $C_{\lambda_n}$ is poly-basic.
\end{prop}

\begin{proof}
  Let $\XX_n$ be the set defined in Definition
  \ref{Minfty}. We can write $\XX_n$ as the filtered colimit of its
  finite subsets $$\XX_n:=\colim_{\lambda_n} \XX_{\lambda_n}.$$
  Set $$C_{\lambda_n}:=\bigoplus_{(C,\phi)\in\XX_{\lambda_n}} C,$$
  and $$P_{\lambda_n}:=C_{\lambda_n}\otimes B.$$ The result now follows.\end{proof}

The next step in the proof is to show that the module $M_{\infty}$
is completely flat. Recall (see, for example, \S $3$ of
\cite{complexity1}) that if $M$ and $N$ are $\Z G$-modules, then
$\hombar_{\Z G}(M,N)$ is the quotient of $\Hom_{\Z G}(M,N)$ by the
additive subgroup consisting of homomorphisms which factor through
a projective module. We have the following useful result (Lemma
$2.3$ in \cite{polyelementary}):

\begin{lemma}\label{hombar lemma}
  Let $M$ and $N$ be $\Z G$-modules. If $M$ is cofibrant, then the
  natural map $$\hombar_{\Z G}(M,N)\ra\exthat^0_{\Z G}(M,N)$$ is
  an isomorphism.
\end{lemma}

We also need the following lemma:

\begin{lemma}\label{cofib mods over a finite gp lemma}
  Let $G$ be a finite group, and $V$ be a $\Z G$-module. Then $V$
  is cofibrant if and only if $V$ is free as a $\Z$-module.
\end{lemma}

\begin{proof}

  Let $B:=B(G,\Z)$ denote the $\Z G$-module of bounded functions from $G$ to $\Z$. First, note that as $G$ is a finite group, $B\cong\Z
  G$.

  Suppose that $V$ is free as a $\Z$-module. Then $V\otimes B\cong
  V\otimes\Z G$ is free as a $\Z G$-module, and hence $V$ is
  cofibrant.

  Conversely, suppose that $V$ is cofibrant, so
  $V\otimes B\cong V\otimes\Z G$ is a projective $\Z G$-module.
  Then $V\otimes\Z G$ is projective as a $\Z$-module, but as $\Z$ is
  a principal ideal domain, every projective $\Z$-module is free.
  Hence, $V\otimes\Z G$ is free as a $\Z$-module, and so it follows
  that $V$ is free as a $\Z$-module.\end{proof}

We can now prove that $M_{\infty}$ is completely flat:

\begin{lemma}\label{completely flat lemma}
  Let $G$ be an $\lhf$-group, and $M$ be any $\Z G$-module. Then
  the module $M_{\infty}$, constructed as in Definition
  $\ref{Minfty}$, is completely flat.
\end{lemma}

\begin{proof}
  This is a generalization of Lemma $4.1$ in
  \cite{polyelementary}:

  As $G$ belongs to $\lhf$, we see from Proposition \ref{reduction to fin subgp
  prop} that it is enough to show that $M_{\infty}$ is completely flat
  over $\Z K$ for all finite subgroups $K$ of $G$. By Lemma \ref{Eckmann--Shapiro Lemma} it is then enough to show that
  $$\exthat^0_{\Z G}(U\otimes_{\Z K}\Z G,M_{\infty})=0$$ for every
  finite subgroup $K$ of $G$ and every completely finitary $\Z
  K$-module $U$.

  Fix $K$ and $U$. As $K$ is finite, $U$ has a complete resolution in the
  sense of \cite{completeres}. Let $V$ be the zeroth kernel in one
  such resolution, so $V$ is a submodule of a projective $\Z
  K$-module. Therefore, $V$ is free as a $\Z$-module and it then follows from Lemma \ref{cofib mods over a finite gp
  lemma} that $V$ is cofibrant as a $\Z K$-module. Then, as $U$ is stably
  isomorphic to $V$, (that is, $U$ and $V$ are isomorphic as objects of the stable category of $\Z K$-modules, defined at       the beginning of this section), it is
  enough to prove that $$\exthat^0_{\Z G}(V\otimes_{\Z K}\Z
  G,M_{\infty})=0.$$ Therefore, we only need to show that $\exthat^0_{\Z
  G}(C,M_{\infty})=0$ for all basic $\Z G$-modules $C$.

  Let $C$ be a basic $\Z G$-module. As $C$ is cofibrant, it
  follows from Lemma \ref{hombar lemma} that the natural map
  $$\hombar_{\Z G}(C,M_{\infty})\ra\exthat^0_{\Z
  G}(C,M_{\infty})$$ is an isomorphism. Let $\phi\in\hombar_{\Z
  G}(C,M_{\infty})$. As $C$ is basic, it is completely finitary,
  and we see that the natural map $$\colim_n\hombar_{\Z
  G}(C,M_n)\ra\hombar_{\Z G}(C,M_{\infty})$$ is an isomorphism.
  Therefore, we can view $\phi$ as an element of
  $\colim_n\hombar_{\Z G}(C,M_n)$, and so $\phi$ is represented by
  some $\widetilde{\phi}\in\hombar_{\Z G}(C,M_n)$ for some $n$.
  Then, as the following diagram commutes: $$\xymatrix{\hombar_{\Z
  G}(C,M_n)\ar[r]\ar[rd] & \colim_n\hombar_{\Z G}(C,M_n)\ar[d]\\
  \blah & \hombar_{\Z G}(C,M_{\infty})}$$ we see that $\phi$ is in
  fact the image of $\widetilde{\phi}$ under the map $$\hombar_{\Z
  G}(C,\iota):\hombar_{\Z G}(C,M_n)\ra\hombar_{\Z
  G}(C,M_{\infty})$$ induced by the natural map $\iota:M_n\ra
  M_{\infty}$.

  The image $\hombar_{\Z G}(C,\iota)(\widetilde{\phi})$ is defined
  as follows: As $\widetilde{\phi}\in\hombar_{\Z G}(C,M_n)$, it is
  represented by some map $\alpha:C\ra M_n$. We can then consider
  the map $$f:C\stackrel{\alpha}{\ra} M_n\stackrel{\iota}{\ra}
  M_{\infty}.$$ Let $\overline{f}$ denote the image of $f$ in
  $\hombar_{\Z G}(C,M_{\infty})$. Then $$\hombar_{\Z
  G}(C,\iota)(\widetilde{\phi}):=\overline{f}.$$

  Now, by construction, we see that the composite $C\ra
  M_n\hookrightarrow M_{n+1}$ factors through the projective
  module $P_n$. Hence, $f$ factors through a projective, and so
  $\overline{f}=0$. We then conclude that $\hombar_{\Z
  G}(C,M_{\infty})=0$, and so $\exthat^0_{\Z G}(C,M_{\infty})=0$,
  and therefore $M_{\infty}$ is completely flat over $\Z G$, as
  required.\end{proof}

Next, recall the following variation on Schanuel's Lemma (Lemma
$3.1$ in \cite{polyelementary}):

\begin{lemma}\label{schanuel}
  Let $$M''\stackrel{\iota}{\mono} M\stackrel{\pi}{\epi} M'$$ be
  any short exact sequence of $R$-modules in which $\pi$ factors
  through a projective module $Q$. Then $M$ is isomorphic to a
  direct summand of $Q\oplus M''$.
\end{lemma}

We now use the fact that the $\Z G$-module $M_{\infty}$ is
completely flat to prove the following:

\begin{prop}\label{divides p plus k prop}
  Let $G$ be an $\lhf$-group and $M$ be a completely finitary,
  cofibrant $\Z G$-module. Then $M$ is isomorphic to a direct
  summand of the direct sum of a poly-basic module and a
  projective module.
\end{prop}

\begin{proof}
  This is a generalization of an argument found in \S$4$ of
  \cite{polyelementary}:

  As in Definition \ref{Minfty}, construct the chain
  $$M=M_0\subseteq M_1\subseteq M_2\subseteq\cdots$$  of $\Z G$-modules, and let $M_{\infty}:=\colim_n M_n$.
  As $G$ belongs to $\lhf$, we see from Lemma \ref{completely flat lemma} that
  $M_{\infty}$ is completely flat, and so $$\exthat^0_{\Z G}(M,M_{\infty})=0.$$
  Also, as $M$ is cofibrant, it follows from Lemma \ref{hombar
  lemma} that $$\hombar_{\Z G}(M,M_{\infty})=0.$$ Then, as $M$ is
  completely finitary, we see that $$\colim_n\hombar_{\Z
  G}(M,M_n)=0.$$ Therefore, there must be some $n$ such that the
  identity map on $M$ maps to zero in $\hombar_{\Z G}(M,M_n)$.
  Hence, we see that the inclusion $M\hookrightarrow M_n$
  factors through a projective module. By Proposition \ref{colimit prop}, we can write $M_n$ as a filtered colimit,
  $$M_n=\colim_{\lambda_{n-1}}\frac{M_{n-1}\oplus
  P_{\lambda_{n-1}}}{C_{\lambda_{n-1}}}:=\colim_{\lambda_{n-1}}
  M_{\lambda_{n-1}},$$ where each
  $P_{\lambda_{n-1}}$ is projective and each $C_{\lambda_{n-1}}$
  is poly-basic. Then, as $M$ is completely finitary, a similar argument to above shows that
  there is some $\lambda_{n-1}$ such that the inclusion $M\hookrightarrow
  M_{\lambda_{n-1}}$ factors through a projective module.

  Now, we can also write $M_{\lambda_{n-1}}$ as a filtered colimit:
  $$M_{\lambda_{n-1}}=\colim_{\lambda_{n-2}}\frac{(\frac{M_{n-2}\oplus
  P_{\lambda_{n-2}}}{C_{\lambda_{n-2}}})\oplus
  P_{\lambda_{n-1}}}{C_{\lambda_{n-1}}}:=\colim_{\lambda_{n-2}}
  M_{\lambda_{n-2}},$$ and we continue as above.

  Continuing in this way, we eventually obtain a map $M\hookrightarrow M_{\lambda_0}$ which
  factors through a projective module $Q$.
  Now, $M_{\lambda_0}$ has been
  constructed in such a way that we have a short exact sequence $$K\mono
  M\oplus P\epi M_{\lambda_0},$$ where
  $P:=P_{\lambda_0}\oplus \cdots\oplus
  P_{\lambda_{n-1}}$, and $K$ admits a filtration $$0=K_{-1}\leq
  K_0\leq \cdots\leq K_{n-1}=K,$$ with each $K_i/K_{i-1}$
  isomorphic to $C_{\lambda_i}$. We see that the second map in the above short
  exact sequence must factor through $P\oplus Q$, and as $K$
  is clearly poly-basic, the result now follows from Lemma
  \ref{schanuel}.\end{proof}

We can now prove the following:

\begin{prop}\label{M divides module with even fg proj res prop}
  Let $G$ be an $\lhf$-group, and $M$ be a completely finitary, cofibrant $\Z G$-module. Then $M$ is
  isomorphic to a direct summand of a $\Z G$-module which has a
  projective resolution that is eventually finitely generated.
\end{prop}

\begin{proof}
  We begin by showing that basic $\Z G$-modules are isomorphic to direct summands of
  $\Z G$-modules with projective resolutions that are eventually finitely generated. Recall that basic $\Z
  G$-modules are of the form $U\otimes_{\Z K}\Z G$, where $K$ is a
  finite subgroup of $G$ and $U$ is a completely finitary,
  cofibrant $\Z K$-module. Write $U$ as the filtered colimit of
  its finitely presented submodules, $$U=\colim_{\lambda}
  U_{\lambda}.$$ As $U$ is completely finitary and cofibrant,
  it follows that $\hombar_{\Z K}(U,-)$ is finitary, and so the natural map
  $$\colim_{\lambda}\hombar_{\Z K}(U,U/U_{\lambda})\ra
  \hombar_{\Z K}(U,\colim_{\lambda} U/U_{\lambda})$$ is an isomorphism; that is,
  $$\colim_{\lambda}\hombar_{\Z K}(U,U/U_{\lambda})=0.$$
  Therefore, there must be some $\lambda$ such that the identity
  map on $U$ maps to zero in $\hombar_{\Z K}(U,U/U_{\lambda})$.
  Hence, we see that the surjection $U\epi U/U_{\lambda}$
  factors through a projective $\Z K$-module $Q$. Then, by Lemma
  \ref{schanuel}, we see that $U$ is isomorphic to a direct summand of $Q\oplus
  U_{\lambda}$. Now, as $K$ is finite, every finitely presented
  $\Z K$-module is of type $\fpinfty$, so in particular
  $U_{\lambda}$ is of type $\fpinfty$. Then, as $U\otimes_{\Z K}\Z G$ is isomorphic to a direct
  summand of $Q\otimes_{\Z K}\Z G\oplus U_{\lambda}\otimes_{\Z
  K}\Z G$, where $Q\otimes_{\Z K}\Z G$ is projective, and $U_{\lambda}\otimes_{\Z K}\Z G$
  is of type $\fpinfty$, we see that $U\otimes_{\Z K}\Z G$ is isomorphic to a
  direct summand of a $\Z G$-module with
  a projective resolution that is eventually finitely
  generated.

  Next, as poly-basic modules are built up from basic modules by
  extensions, we see from the Horseshoe Lemma that every
  poly-basic $\Z G$-module is isomorphic to a direct summand of a $\Z G$-module
  with a projective resolution that is eventually finitely
  generated.

  Finally, if $G$ is an $\lhf$-group, and $M$ is a completely finitary,
  cofibrant $\Z G$-module, it follows from Proposition \ref{divides p plus k
  prop} that $M$ is isomorphic to a direct summand of $P\oplus C$, for some projective
  module $P$ and some poly-basic module $C$. Then, as $C$ is isomorphic to a direct summand of a $\Z G$-module with a projective
  resolution that is eventually finitely generated, the result now
  follows.\end{proof}

We now have the following proposition:

\begin{prop}\label{M tensor B has finite pd prop}
  Let $G$ be an $\lhf$-group, and $M$ be a completely finitary $\Z
  G$-module. Also, let $B:=B(G,\Z)$ denote the $\Z G$-module of bounded functions
  from $G$ to $\Z$. Then $M\otimes B$ has finite projective dimension
  over $\Z G$.
\end{prop}

\begin{proof}
  This is a generalization of Proposition $9.2$ in
  \cite{strgpgrdrings}:

  Let $K$ be a finite subgroup of $G$.  We see from the Proposition in \cite{KropTalelli} that $B$ is free as a $\Z K$-module, so $M\otimes B$ is
  a direct sum of copies of $M\otimes\Z K$ as a $\Z K$-module, and hence has finite projective dimension
  over $\Z K$. It then follows from Lemma $4.2.3$ in
  \cite{Hierarch} that $$\exthat^0_{\Z K}(A,M\otimes B)=0$$ for
  any $\Z K$-module $A$. In particular, we see that $M\otimes B$
  is completely flat over $\Z K$. As this holds for any finite
  subgroup $K$ of $G$, we see from Proposition \ref{reduction to fin subgp
  prop} that $M\otimes B$ is completely flat over $\Z G$. Then, as
  $M$ is completely finitary over $\Z G$, we see that
  $$\exthat^0_{\Z G}(M,M\otimes B)=0,$$ and it then follows from Lemma $2.2$ in
  \cite{Homfinconditions} that $M\otimes B$ has finite projective
  dimension over $\Z G$.\end{proof}

\begin{lemma}\label{exists cofib cfin module}
  Let $G$ be an $\lhf$-group with cohomology almost everywhere finitary. Then there is an integer $n\geq 0$ such that in any projective resolution $P_*\epi\Z$ of $G$ the $n$th kernel is a completely finitary, cofibrant module.
\end{lemma}

\begin{proof}
  As $G$ has cohomology almost everywhere finitary, it follows from $4.1$(ii) in \cite{fp} that the trivial $\Z G$-module $\Z$ and every kernel of a projective resolution of $G$ is completely finitary. By Proposition \ref{M tensor B has finite pd prop} it follows that $B$ has finite projective dimension over $\Z G$. If $\pd_{\Z G} B=n$, then clearly the $n$th kernel of any projective resolution of $G$ is cofibrant.
\end{proof}

Next, we have two straightforward results:

\begin{prop}\label{final 2 mods even fg prop}
  Let $R$ be a ring, and suppose that $$0\ra N'\ra N\ra P_n\ra\cdots\ra P_0\ra M\ra
  0$$ is an exact sequence of $R$-modules such that
  the $P_i$ are projective, and $N'$ and $N$ have projective resolutions that
  are eventually finitely generated. Then the partial projective
  resolution $$P_n\ra \cdots\ra P_0\ra M\ra 0$$ of $M$ can be extended to
  a projective resolution that is eventually finitely generated.
\end{prop}

\begin{proof}
  Let $K:=\ker(P_n\ra P_{n-1})$, so we have the following short exact
  sequence: $$N'\mono N\epi K.$$ Next, let $Q_*\epi N$ be a projective
  resolution of $N$ that is eventually finitely generated, and
  let $L$ denote the zeroth kernel. We then have the following: $$\xymatrix{\widetilde{K}\ar@{>-->}[rd] & L\ar@{>->}[d] & \blah\\
  \blah & Q_0\ar@{->>}[d]\ar@{-->>}[rd] & \blah\\ N'\ar@{>->}[r] &
  N\ar@{->>}[r] & K}$$ where $\widetilde{K}$ is an extension of
  $N'$ by $L$, and since both $N'$ and $L$ have projective
  resolutions that are eventually finitely generated, it follows
  from the Horseshoe Lemma that
  $\widetilde{K}$ also has such a resolution. We then have the
  following exact sequence:
  $$0\ra\widetilde{K}\ra Q_0\ra P_n\ra\cdots\ra P_0\ra M\ra 0,$$
  and the result now follows.\end{proof}

\begin{prop}\label{proj res eventualy fg implies free res
eventually fg prop}
  Let $M$ be an $R$-module. If $M$ has a projective resolution
  that is eventually finitely generated, then $M$ has a free
  resolution that is eventually finitely generated.
\end{prop}

\begin{proof}
  Let $P_*\epi M$ be a projective resolution of $M$ that is
  eventually finitely generated; say $P_j$ is finitely generated
  for all $j\geq n$, and let $$K:=\ker(P_{n-1}\ra P_{n-2}).$$ Then $K$
  is of type $\fpinfty$, and hence of type $\flinfty$. We can therefore choose a free resolution $F_{n+*}\epi K$ of $K$ with all
  the free modules finitely generated. This gives the following
  exact sequence: $$\cdots\ra F_{n+1}\ra F_n\ra
  P_{n-1}\ra\cdots\ra P_2\ra P_1\ra P_0\ra M\ra 0.$$

  Next, recall the Eilenberg trick (Lemma $2.7$ \S VIII in
  \cite{Brown}): For any projective $R$-module $P$, we can
  choose a free $R$-module $F$ such that $P\oplus F\cong F$.
  Therefore, using this, we can replace the projective modules
  $P_i$ in the above exact sequence by free modules $F_i$, at the
  expense of changing $F_n$ to a larger free module $F_n'$. We
  then have the following free resolution $$\cdots\ra F_{n+2}\ra
  F_{n+1}\ra F_n'\ra F_{n-1}\ra\cdots\ra F_0\ra M\ra 0$$ of $M$,
  with the $F_j$ finitely generated for all $j\geq n+1$.\end{proof}

We now have the following proposition (Proposition $5.1$
in \cite{Deltapaper}):

\begin{prop}\label{Kropholler prop}
  Let $X^n$ be an $(n-1)$-connected $n$-dimensional
  $G$-CW-complex, where $n\geq 2$. Let $\phi:F\ra H_n(X^n)$ be a
  surjective $\Z G$-module map from a free $\Z G$-module $F$ to
  the $n$th homology of $X^n$. Then $X^n$ can be embedded into an
  $n$-connected $(n+1)$-dimensional $G$-CW-complex $X^{n+1}$ such
  that $G$ acts freely outside $X^n$ and there is a short exact
  sequence $$0\ra H_{n+1}(X^{n+1})\ra F\ra H_n(X^n)\ra 0.$$
\end{prop}

Finally, we can now prove the implication (i) $\Rightarrow$ (ii)
of Theorem A.

\begin{theorem}
  Let $G$ be an $\lhf$-group with cohomology almost everywhere
  finitary. Then $G\times\Z$ has an Eilenberg--Mac Lane space $K(G\times\Z,1)$ with
  finitely many $n$-cells for all sufficiently large $n$.
\end{theorem}

\begin{proof}
  Let $Y$ be the $2$-complex associated to some presentation of
  $G$, and let $\widetilde{Y}$ denote its universal cover. The augmented cellular chain complex of $\widetilde{Y}$ is a
  partial free resolution of the trivial $\Z G$-module, which we
  denote by $$F_2\ra F_1\ra F_0\ra\Z\ra 0.$$ We can extend this to
  a free resolution $F_*\epi\Z$ of the trivial $\Z G$-module, and as $G$ is an $\lhf$-group with cohomology almost
  everywhere finitary, it follows from Lemma \ref{exists cofib cfin
  module} that there is some $m\geq 0$ such that the $m$th kernel
  $$M:=\ker(F_{m-1}\ra F_{m-2})$$ of this resolution is a completely finitary,
  cofibrant $\Z G$-module. We then have the following exact sequence of $\Z G$-modules: $$0\ra M\ra F_{m-1}\ra\cdots\ra F_0\ra\Z\ra 0.$$

  Next, recall that the
  circle $S^1$ is an Eilenberg--Mac Lane space $K(\Z,1)$,
  with universal cover $\R$. The augmented cellular chain
  complex of $\R$ is the following free
  resolution of the trivial $\Z \Z$-module: $$0\ra\Z
  \Z\ra\Z \Z\ra\Z\ra 0.$$

  If we tensor these two exact sequences together, we obtain the following
  exact
  sequence of $\Z[G\times\Z]$-modules: $$\begin{array}{lrrrr}
     0\ra
  M\otimes\Z \Z\ra M\otimes\Z \Z\oplus
  F_{m-1}\otimes\Z \Z\ra & \blah & \blah & \blah & \blah \end{array}$$ $$\begin{array}{llr} \blah & \blah & F_{m-1}\otimes\Z \Z\oplus
  F_{m-2}\otimes\Z \Z \ra \cdots\ra F_0\otimes\Z
  \Z\ra\Z\ra 0.
  \end{array}$$

  Now, as $M$ is a completely finitary, cofibrant $\Z G$-module,
  it follows from Proposition \ref{M divides module with even fg proj res
  prop} that $M$ is isomorphic to a direct summand of some $\Z G$-module $L$ which has a
  projective resolution that is eventually finitely generated. We then obtain the following exact sequence of $\Z
  [G\times \Z]$-modules: $$\begin{array}{lrrrr}
     0\ra
  L\otimes\Z \Z\ra L\otimes\Z \Z\oplus
  F_{m-1}\otimes\Z \Z\ra & \blah & \blah & \blah & \blah \end{array}$$ $$\begin{array}{llr} \blah & \blah & F_{m-1}\otimes\Z \Z\oplus
  F_{m-2}\otimes\Z \Z \ra \cdots\ra F_0\otimes\Z
  \Z\ra\Z\ra 0.
  \end{array}$$

  It now follows from Propositions \ref{final 2 mods even fg prop}
  and \ref{proj res eventualy fg implies free res eventually fg
  prop} that we can extend the partial free resolution
  $$F_{m-1}\otimes\Z \Z\oplus F_{m-2}\otimes\Z
  \Z\ra\cdots\ra F_0\otimes\Z\Z\ra \Z\ra 0$$ of the trivial $\Z [G\times\Z]$-module to a
  free resolution that is eventually finitely generated. We shall denote
  this free resolution by $F'_*\epi\Z$.

  Next, let $X^2$ denote the subcomplex of $\widetilde{Y}\times \R$,
  consisting of the $0$, $1$ and $2$-cells. Then, as
  $$C_*(\widetilde{Y}\times\R)\cong C_*(\widetilde{Y})\otimes C_*(\R),$$ we see that the augmented cellular
  chain complex of $X^2$ is the following: $$F_2'\ra F_1'\ra F_0'\ra\Z\ra
  0,$$ and, furthermore,
  that $\widetilde{H}_i(X^2)=0$ for $i=0,1$. We therefore have the
  following exact sequence: $$0\ra \widetilde{H}_2(X^2)\ra
  F'_2\ra F'_1\ra F'_0\ra\Z \ra 0,$$ and as $F'_3\epi
  \widetilde{H}_2(X^2)$, it follows from Proposition \ref{Kropholler prop} that we can embed $X^2$ into a
  $2$-connected $3$-complex $X^3$ such that we have the following short exact
  sequence: $$0\ra \widetilde{H}_3(X^3)\ra F_3'\ra \widetilde{H}_2(X^2)\ra 0.$$
  Then $F_4'\epi\widetilde{H}_3(X^3)$, and we can continue as
  before.

  By induction, we can then construct a space, which we denote by $X$, such that
  $C_n(X)=F_n'$ for all $n$. Then, as the free resolution
  $F_*'\epi\Z$ is eventually finitely generated, it follows that
  $C_n(X)$ is finitely generated for all sufficiently large $n$. Also, we see that $\widetilde{H}_i(X)=0$ for all $i$, and so $X$ is
  contractible (see \S I.$4$ in \cite{Brown}).

  We see from Proposition $1.40$ in \cite{Hatcher} that $X$ is the
  universal cover for the quotient space
  $\overline{X}:=X/G\times \Z$, and furthermore that
  $\overline{X}$ has fundamental group isomorphic to $G\times
  \Z$. Thus, $\overline{X}$ is an
  Eilenberg--Mac Lane space $K(G\times \Z,1)$, and as $C_n(X)$ is finitely generated for all
  sufficiently large $n$, we conclude that $\overline{X}$ has finitely many $n$-cells for all
  sufficiently large $n$, as required.\end{proof}

\subsection{Proof of Theorem A (ii) $\Rightarrow$ (iii)}$ $

We do not require the assumption that $G$ belongs to $\lhf$ for
this section.

Recall from page $528$ of \cite{Hatcher} that a space $Y$ is said
to be \emph{dominated} by a space $K$ if and only if $Y$ is a
retract of $K$ in the homotopy category; that is, there are maps
$i:Y\ra K$ and $r:K\ra Y$ such that $ri\simeq \id_Y$.

\begin{prop}
  Suppose that $K$ is a $K(G\times\Z,1)$ space with finitely many
  $n$-cells for all sufficiently large $n$. Then $G$ has an
  Eilenberg--Mac Lane space $K(G,1)$ which is dominated by $K$.
\end{prop}

\begin{proof}
  As every group has an Eilenberg--Mac Lane space (Theorem $7.1$ \S VIII in
  \cite{Brown}), we can choose a $K(G,1)$ space $Y$.
  Then, as $S^1$ is a $K(\Z,1)$ space, we see from Example $1$B.$5$ in
  \cite{Hatcher} that $Y\times S^1$ is a $K(G\times\Z,1)$ space. Then,
  as $K(G\times\Z,1)$ spaces are unique up to homotopy equivalence
  (Theorem $1$B.$8$ in \cite{Hatcher}), we see that $Y\times
  S^1\simeq K$, and hence that $Y$ is dominated by $K$.\end{proof}

\subsection{Proof of Theorem A (iii) $\Rightarrow$ (i)}$ $

Once again, we do not require the assumption that $G$ belongs to
$\lhf$ for this section.

\begin{lemma}\label{replace with one with same fun gp lemma}
  Let $Y$ be a $K(G,1)$ space which is
  dominated by a CW-complex with finitely many cells in all sufficiently high
  dimensions. Then we may choose this complex to have fundamental group
  isomorphic to $G$.
\end{lemma}

\begin{proof}
  Let $Y$ be dominated by a CW-complex $K$ that has finitely many
  cells in all sufficiently high dimensions, so there are maps $$Y\stackrel{i}{\ra}
  K\stackrel{r}{\ra} Y$$ such that $ri\simeq\id_Y$. Applying $\pi_1$ gives maps
  $$\pi_1(Y)\stackrel{\pi_1(i)}{\ra}\pi_1(K)\stackrel{\pi_1(r)}{\ra}
  \pi_1(Y)$$ such that $\pi_1(r)\pi_1(i)=\id_{\pi_1(Y)}$. Hence,
  $\pi_1(r)$ is surjective. Let $K'$ denote the kernel of
  $\pi_1(r)$, and let $W$ be a bouquet of circles, with one circle for
  each generator in some chosen presentation of $K'$, so there is an
  obvious map $W\ra K$.

  Next, let $CW$ denote the cone on $W$, and form
  the following pushout: $$\xymatrix{W\ar[r]\ar[d] & K\ar@{-->}[d]\\
  CW\ar@{-->}[r] & L}$$ It follows that $L$ is a CW-complex with
  finitely many cells in all sufficiently high dimensions.

  Now, the composite map $W\ra
  K\stackrel{r}{\ra} Y$ is clearly nullhomotopic, and therefore
  lifts through the cone, so we have the following diagram: $$\xymatrix{W\ar[r]\ar[d] & K\ar[d]\ar@/^/[rdd]^r & \blah\\
  CW\ar[r]\ar@/_/[rrd]  & L & \blah\\ \blah & \blah & Y}$$ and so
  by the definition of pushout, there is an induced map $L\ra Y$
  making the above diagram commute. If we now compose this with the
  map $Y\stackrel{i}{\ra} K\ra L$, we obtain a map $Y\ra
  L\ra Y$ that is homotopic to the identity on $Y$. Hence, $Y$ is
  dominated by $L$.

  Finally, by
  van Kampen's Theorem (Theorem $1.20$ in
  \cite{Hatcher}), we see that $$\begin{array}{lcl}
    \pi_1(L) & \cong & \pi_1(K)/\im(\pi_1(W)\ra\pi_1(K)) \\
    \blah & \cong & \pi_1(Y)\\ \blah &\cong & G,
  \end{array}$$ as required.\end{proof}

Next, recall from \cite{BrownPaper} that if $P:=(P_i)_{i\geq 0}$
is a chain complex of projective $\Z G$-modules, then we define
the \emph{cohomology theory} $H^*(P,-)$ \emph{determined by $P$}
as
$$H^n(P,M):=H^n(\Hom_{\Z G}(P_*,M))$$ for every
$\Z G$-module $M$ and every $n\in\N$.

\begin{lemma}\label{even fg proj lemma}
  Let $P:=(P_i)_{i\geq 0}$ be a chain complex of projective
  $\Z G$-modules. If $P_{n-1}, P_n$ and $P_{n+1}$ are finitely
  generated, then $H^n(P,-)$ is finitary.
\end{lemma}

\begin{proof}
  Firstly, recall from Lemma $4.7$ \S VIII in
  \cite{Brown} that if $Q$ is a finitely generated projective
  module, then the functor $\Hom_{\Z G}(Q,-)$ is finitary.

  Next, let $M:=\coker(P_{n+1}\ra P_n)$, so we have the following
  exact sequence: $$P_{n+1}\ra P_n\ra M\ra 0,$$ which
  gives the following exact sequence of functors:
  $$0\ra\Hom_{\Z G}(M,-)\ra\Hom_{\Z G}(P_n,-)\ra\Hom_{\Z
  G}(P_{n+1},-).$$ Let $(N_{\lambda})$ be any filtered colimit
  system of $\Z G$-modules, so we have the following commutative
  diagram with exact rows: $$\xymatrix{0\ar[r]\ar[d] & \colim_{\lambda}\Hom_{\Z
  G}(M,N_{\lambda})\ar[r]\ar[d] & \colim_{\lambda}\Hom_{\Z
  G}(P_n,N_{\lambda})\ar[r]\ar[d] & \colim_{\lambda}\Hom_{\Z
  G}(P_{n+1},N_{\lambda})\ar[d]\\0\ar[r] & \Hom_{\Z
  G}(M,\colim_{\lambda} N_{\lambda})\ar[r] & \Hom_{\Z
  G}(P_n,\colim_{\lambda} N_{\lambda})\ar[r] & \Hom_{\Z
  G}(P_{n+1},\colim_{\lambda} N_{\lambda})}$$ and as both
  $\Hom_{\Z G}(P_n,-)$ and $\Hom_{\Z G}(P_{n+1},-)$ are finitary,
  the two right-hand maps are isomorphisms. It then follows from
  the Five Lemma that the natural map $$\colim_{\lambda} \Hom_{\Z
  G}(M,N_{\lambda})\ra\Hom_{\Z G}(M,\colim_{\lambda}
  N_{\lambda})$$ is an isomorphism, and hence that $\Hom_{\Z
  G}(M,-)$ is finitary.

  Then, as we have the following exact sequence of functors:
  $$\Hom_{\Z G}(P_{n-1},-)\ra\Hom_{\Z G}(M,-)\ra H^n(P,-)\ra 0,$$ the
  result now follows from another application of the Five Lemma.\end{proof}

We can now prove the implication (iii) $\Rightarrow$ (i) of
Theorem A:

\begin{prop}
  Suppose that $G$ has an Eilenberg--Mac Lane space $K(G,1)$ which is dominated by a CW-complex with finitely many $n$-cells
  for all sufficiently large $n$. Then $G$ has cohomology almost
  everywhere finitary.
\end{prop}

\begin{proof}
  This is a generalization of the proof of Proposition $6.4$ \S VIII
  in \cite{Brown}:

  Let $Y$ be such a $K(G,1)$ space. By Lemma \ref{replace with one with same fun gp lemma}, we see that
  $Y$ is dominated by a CW-complex $K$ with finitely many cells in
  all sufficiently high dimensions, such that $K$ has fundamental
  group isomorphic to $G$. Let
  $\widetilde{Y}$ and $\widetilde{K}$ denote the respective
  universal covers. We see that
  $C_*(\widetilde{Y})$ is a retract of $C_*(\widetilde{K})$ in the
  homotopy category of chain complexes over $\Z G$.
  Therefore, we obtain maps giving the following commutative diagram: $$\xymatrix{H^*(G,-)\ar@{=}[rd]\ar[r] &
  H^*(C,-)\ar[d]\\ \blah & H^*(G,-)}$$ where $H^*(C,-)$ denotes the cohomology
  theory determined by $C_*(\widetilde{K})$. We then conclude that
  $H^*(G,-)$ is a direct summand of $H^*(C,-)$.

  Now, as $K$ has finitely many cells in all sufficiently high dimensions, it
  follows that $C_*(\widetilde{K})$ is eventually finitely
  generated, and so by Lemma \ref{even fg proj lemma} that $H^k(C,-)$ is
  finitary for all sufficiently large $k$. The result then follows
  from an application of the Five Lemma.\end{proof}

\end{document}